\documentclass[12pt]{article}
\usepackage{amsmath, amsthm}
\usepackage{amscd,amssymb,amsmath,amsthm}
\usepackage{mathrsfs}
\usepackage[colorlinks=true,linkcolor=blue,citecolor=red]{hyperref}
\newtheoremstyle{theorem}
  {10pt}		  
  {10pt}  
  {\sl}  
  {\parindent}     
  {\bf}  
  {. }    
  { }    
  {}     
\theoremstyle{theorem}
\newtheorem{theorem}{Theorem}[section]
\newtheorem{corollary}{Corollary}[section]
\newtheorem{lemma}{Lemma}[section]

\newtheorem{definition}{Definition}[section]
\newtheoremstyle{defi}
  {10pt}		  
  {10pt}  
  {\rm}  
  {\parindent}     
  {\bf}  
  {. }    
  { }    
  {}     
\theoremstyle{defi}

\numberwithin{equation}{section}

\usepackage{hyperref}

\numberwithin{equation}{section}

\def \beqs {\begin{eqnarray*}}
\def \eds {\end{eqnarray*}}
\def \beq {\begin{eqnarray}}
\def \ed {\end{eqnarray}}

\numberwithin{equation}{section}
\usepackage[displaymath, mathlines]{lineno}
\begin{document}
\title{\bf Certain Class of Analytic\\ Functions Based on \\$q$-difference operator}
\author{ K.VIJAYA
\\School of Advanced Sciences,\\ VIT University,
Vellore, India- 632 014.\\
{\bf E-mail: kvijaya@vit.ac.in} }
\date{\vspace{-5ex}}
\maketitle
\begin{abstract}
In this paper, we considered a generalized class
of starlike functions  defined by Kanas and R\u{a}ducanu\cite{10}  to obtain integral means inequalities and subordination results.Further, we obtain the
for various subclasses of starlike
functions.\end{abstract}



\section{Introduction}
\label{sec-1}

Let $\mathcal{A}$ denote the class of functions of the form
\begin{equation}\label{e1.1}
f(z) = z + \sum\limits_{n=2}^{\infty} a_nz^n
\end{equation}
which   are analytic in the open disc $\mathbb{U} = \{z :\, |z|<1 \}.$
Also denote by $ \mathcal{T} $ a subclass of $\mathcal{A}$ consisting
functions of the form \begin{equation}\label{e1.8}
 f(z) = z -
\sum\limits_{n=2}^{\infty} a_n z^n,\,\, a_n \geq 0,\,\, z \in \mathbb{U}
\end{equation}
 introduced and studied  new subclasses namely ,the class of
starlike functions of order $\alpha$ denoted by $\mathcal{T}^*(\alpha)$ and
convex functions of order $\alpha$ denoted by $\mathcal{K}(\alpha)$ of $\mathcal{T}$ in\cite{ref13}.
 \par Now, we refer  to a notion of \textit{q-operators} i.e.
\textit{q-difference operator}  that
play vital role in the theory of hypergeometric series,  quantum
physics and in the operator theory. The application of q-calculus
was initiated by Jackson \cite{fhjj} (also see \cite{aga,prr,10} ).
Kanas and R\u{a}ducanu \cite{10} have used the fractional \textit{q-calculus
operators }in investigations of certain classes of functions which
are analytic in $\mathbb{U}.$

\par For $0<q<1$  the Jackson's \textit{q-derivative}  of a function
$f\in \mathcal{A}$ is, by definition,  given as follows \cite{fhjj}
\begin{equation}\label{in5}
D_{q}f(z) = \left\{\begin{array}{lcl}\dfrac{f(z) - f(qz)}{(1 -
q)z}&for &z \neq 0,\\ f'(0) &for& z=0,\end{array}\right.
\end{equation}
and $D^{2}_{q}f(z) = D_{q}(D_{q}f(z)).$ From (\ref{in5}), we have
$ D_{q}f(z)=1+\sum\limits_{n=2}^{\infty}[n]a_{n}z^{n-1},$
where
\begin{equation}\label{in6}
[n] = \frac{1 - q^{n}}{1 - q},
\end{equation}
is sometimes called \textit{the basic number} $n$. If
$q\rightarrow1^{-}, [n]\rightarrow n$. For a function $h(z) =
z^{m},$ we obtain
$D_{q}h(z)=D_{q}z^{m}=\frac{1 - q^{m}}{1 - q}z^{m-1}=[m]z^{m-1},$ and
$\lim_{q\rightarrow 1^-}D_{q}h(z)=\lim_{q\rightarrow 1^-}\left([m]z^{m-1}\right)= mz^{m-1} = h'(z), $
where $h'$ is the ordinary derivative.
Let $t\in \mathbb{R}$ and $n\in \mathbb{N}.$ Let the $q-$generalized Pochhammer symbol be defined as
$ [t]_n=[t][t+1][t+2]...[t+n-1]$ and for $t>0$ let the $q-gamma$ function be defined by
\begin{equation}\label{qgamma}
    \Gamma_q(t+1)=[t] \Gamma_q(t) \qquad \qquad \Gamma_q(1)=1.
\end{equation}
 Now we recall the definition of Ruscheweyh \textit{q-differential operator} defined and discussed by Kannas and Raducanu\cite{10}.

 \begin{definition}\cite{10} For $f\in \mathcal{A}$ let the Ruscheweyh \textit{q-differential operator} be

 \begin{equation}\label{rq1}
    R^\lambda_q f(z)=f(z)*F_{q.\lambda+1}(z) \quad (z\in \mathbb{U}, \lambda > 1)
 \end{equation}
 where
 \begin{equation}\label{rq2}
   F_{q,\lambda+1}(z)=z+ \sum\limits_{n=2}^{\infty} \frac{\Gamma_q(n+\lambda)}{[n-1]!\Gamma_q(1+\lambda)}z^n= z+\sum\limits_{n=2}^{\infty}\frac{[\lambda+1]_{n-1}}{[n-1]!}z^n
 \end{equation} and $*$ stands for the Hadanard product (or convolution).
 \end{definition}
 From \eqref{rq1} we note that $R^0_q f(z)=f(z); R^1_q f(z)=zD_q f(z)$ and $R^1_q f(z)=\frac{zD_q^m(z^{m-1}f(z))}{[m]!}.$
 Making use of \eqref{rq1} and \eqref{rq2}, we have
 \begin{equation}\label{rq3}
    R^\lambda_q f(z)=z+ \sum\limits_{n=2}^{\infty} \frac{\Gamma_q(n+\lambda)}{[n-1]!\Gamma_q(1+\lambda)}a_nz^n\quad (z\in\mathbb{U})
 \end{equation}
  Note that  as $\lim_{q\rightarrow 1^-}$ we get $$F_{q,\lambda+1}(z)=\frac{z}{(1-z)^{\lambda+1}}{~\rm and~ } R^\lambda_q f(z)= f(z)*\frac{z}{(1-z)^{\lambda+1}}.$$
As a consequence of \eqref{in5}, for $f\in \mathcal{A}$ we
obtain
\begin{equation}\label{rq4}
 D_{q} (R^\lambda_q f(z)) = 1 + \sum_{n=2}^\infty [n]\frac{\Gamma_q(n+\lambda)}{[n-1]!\Gamma_q(1+\lambda)}a_{n}z^{n-1}.
\end{equation}
where
\begin{equation}\label{e1.7a}
  \Psi_q(n,\lambda)= \frac{\Gamma_q(n+\lambda)}{[n-1]!\Gamma_q(1+\lambda)}
\end{equation}
Now  we recall the following definition of the function class $\mathcal{ST}^\lambda_q (k,\alpha)$ and its characterization property due to  Kannas and Raducanu \cite{10}.
\par For $0 \leq \alpha <1,$ $k\geq 0$ and  $\lambda \geq -1,$ let
$\mathcal{ST}^\lambda_q (k,\alpha)$ be the subclass of
$\mathcal{A}$ consisting of functions of the form (\ref{e1.1}) and
satisfying the analytic criterion
\begin{equation}\label{e1.7}
\mbox{Re}~ \left (  \frac{z\mathcal{D}_q (R^\lambda_q f(z))
}{R^\lambda_q f(z)}-\alpha \right)
> k \left |  \frac{z\mathcal{D}_q (R^\lambda_q f(z))
}{R^\lambda_q f(z)} - 1 \right |,~~~ z \in \mathbb{U},
\end{equation}
where $ \mathcal{D}_q (R^\lambda_q f(z)$ is given by (\ref{rq4}).

\begin{theorem}\label{tb1.1}\cite{10}
Let $f(z)\in \mathcal{A}$ of the form (\ref{e1.1}). If the equality
\begin{equation}\label{e1.13}
\sum\limits_{n=2}^{\infty}  ([n](1+k) - k -\alpha)
]\Psi_q(n,\lambda)~|a_n| \leq 1-\alpha,
\end{equation}
holds true for some $0\leq k <\infty,$  $\lambda \geq -1,$ $0 \leq \alpha <1,$
 and $\Psi_q(n,\lambda)$ is given by (\ref{e1.7a}), then $f\in\mathcal{ST}^\lambda_q (k,\alpha).$
 The result is sharp for  the function
 \begin{equation}\label{sharp}
f_n(z) = z - \frac{1 - \alpha}{\Phi_n(\lambda,\alpha, k)}z^n,
 \end{equation}
 where
 \begin{equation}\label{phi01}
\Phi_n(\lambda,\alpha, k)=([n](1+k) - k -\alpha)
]\Psi_q(n,\lambda)
 \end{equation} and  $\Psi_q(n,\lambda)$ given by \eqref{e1.7a}.
\end{theorem}

It is interest to note that in \cite{10}, Kannas and Raducanu, state that the condition \eqref{e1.13} is also necessary
 for functions 
 $f_n$ given by \eqref{sharp}  is in
     $\mathcal{ST}^\lambda_q (k,\alpha).$

\par  Making use of the above necessary and sufficient conditions for$ f\in \mathcal{ST}^\lambda_q (k,\alpha),$ 
in this paper, we obtain integral means inequalities  and subordination results.
\section{Integral Means Inequalities}

\begin{definition}\label{d1}{\em (Subordination Principle)}
For analytic functions $g$ and $h$ with $g(0)=h(0),$ $g$ is said
to be subordinate to $h,$ denoted by $g \prec h,$ if there exists
an analytic function $w$ such that $w(0)=0,$ $|w(z)|<1$ and $g(z)
= h(w(z)),$ for all $z \in \mathbb{U}.$
\end{definition}
\begin{lemma}\label{llittle} \cite{Little} If the functions $f$ and
$g$ are analytic in $\mathbb{U}$ with $g \prec f,$ then for $\eta > 0,$ and
$0 < r < 1,$
\begin{equation}\label{elittle}
\int\limits_{0}^{2\pi} \left | g(re^{i\theta}) \right |^{\eta }
d\theta \leq \int\limits_{0}^{2\pi} \left | f(r e^{i\theta})
\right |^{\eta } d\theta.
\end{equation}
\end{lemma}
\par In \cite{ref13}, Silverman found that the function $f_2(z) = z -
\frac{z^2}{2}$ is often extremal over the family $\mathcal{T}$ and applied
this function to resolve his integral means inequality,
conjectured in \cite{Silver91} and settled in \cite{Silver97},
that
\[
\int\limits_{0}^{2\pi} \left | f(re^{i\theta}) \right |^{\eta }
d\theta \leq \int\limits_{0}^{2\pi} \left | f_2(r e^{i\theta})
\right |^{\eta } d\theta,
\]
for all $f\in \mathcal{T},$ $\eta > 0$ and $0 < r <1.$ In \cite{Silver97},
Silverman also proved his conjecture for the subclasses $\mathcal{T}^*(\alpha)$ and
$\mathcal{K}(\alpha)$ of $\mathcal{T}.$

\par Applying Lemma \ref{llittle} and  Theorem \ref{tb1.1} , we prove the
following result.
\begin{theorem}\label{t4.1}
Suppose $f \in \mathcal{ST}^\lambda_q (k,\alpha),$ $\eta >
0,$ $0\leq k <\infty,$  $\lambda \geq -1,$ $0 \leq \alpha <1,$ and
$f_2(z)$ is defined by
$
f_2(z) = z - \frac{1 - \alpha}{\Phi_2(\lambda, \alpha, k)}z^2,
$
where
\begin{equation}\label{phi}
\Phi_2(\lambda,\alpha, k)=([2](1+k) - k -\alpha)
]\Psi_q(2,\lambda)
\end{equation}
and from \eqref{e1.7a}, $\Psi_q(2,\lambda) $ is given by \begin{equation}\label{e2.3}
\Psi_q(2,\lambda) = \frac{\Gamma_q(2+\lambda)}{[2-1]!\Gamma_q(1+\lambda)}
\end{equation}. Then for
$z=re^{i\theta},$ $0 < r < 1,$ we have
\begin{equation}\label{e3.1}
\int\limits_{0}^{2\pi} \left | f(z) \right |^{\eta } d\theta \leq
\int\limits_{0}^{2\pi} \left | f_2(z) \right |^{\eta} d\theta .
\end{equation}
\end{theorem}
\begin{proof}
For $f(z) = z - \sum\limits_{n=2}^{\infty}|a_n| z^n,$ (\ref{e3.1})
is equivalent to proving that
\[
\int\limits_{0}^{2\pi} \left | 1 - \sum\limits_{n=2}^{\infty}
|a_n| z^{n-1} \right |^{\eta } d\theta \leq \int\limits_{0}^{2\pi}
\left |1 - \frac{1-\alpha}{ \Phi_2(\lambda,\alpha, k)} z \right
|^{\eta } d\theta .
\]
By Lemma \ref{llittle}, it suffices to show that
\[
1 - \sum\limits_{n=2}^{\infty} |a_n| z^{n-1} \prec 1 -
\frac{1-\alpha}{\Phi_2(\lambda,\alpha, k)}z.
\]
Setting
\begin{equation}\label{e3.2}
1 - \sum\limits_{n=2}^{\infty} |a_n| z^{n-1} = 1 -
\frac{1-\alpha}{\Phi_n(\lambda, \alpha, k)}w(z),
\end{equation}
and using (\ref{e1.13}), we obtain $w(z)$ is analytic in $\mathbb{U}, w(0)
= 0,$ and
\begin{align*}
|w(z)| & = \left | \sum\limits_{n=2}^{\infty}
\frac{\Phi_2(\lambda,\alpha, k))}{1-\alpha} |a_n| z^{n-1}\right |\\
       & \leq |z| \sum\limits_{n=2}^{\infty}\frac{\Phi_n(\lambda,\alpha, k)}{1-\alpha}|a_n| \\
       & \leq |z|,
\end{align*}
where $\Phi_n(\lambda,\alpha, k)$  is as given in  \eqref{phi01}. This completes the proof by Theorem\,\ref{t4.1}.
\end{proof}
\section{Subordination Results}
\par Following Frasin \cite{BAF} and Singh \cite{singh} , we obtain subordination results for the new
class $\mathcal{ST}^\lambda_q (k,\alpha).$

\begin{definition}\label{d2}{\em (Subordinating Factor Sequence)}
A sequence $\{b_n\}^{\infty}_{n=1}$ of complex numbers is said to
be a subordinating sequence if, whenever
$f(z)=\sum\limits_{n=1}^{\infty}a_nz^n,$ $a_1=1$ is regular,
univalent and convex in $U,$ we have
\begin{equation}\label{e1.11}
\sum\limits_{n=1}^{\infty}b_na_nz^n \prec f(z),\,\,\,\, z \in \mathbb{U}.
\end{equation}
\end{definition}
\begin{lemma}\label{l1.1}\cite{wilf}
The sequence $\{b_n\}^{\infty}_{n=1}$ is a subordinating factor
sequence if and only if
\begin{equation}\label{e1.12}
\Re\left (1 + 2 \sum\limits_{n=1}^{\infty}b_nz^n\right )
> 0, \,\,\,\, z \in \mathbb{U}.
\end{equation}
\end{lemma}

\begin{theorem}\label{t2.1}
Let $ f \in \mathcal{ST}^\lambda_q (k,\alpha)$ and $g(z)$ be
any function in the usual class of convex functions $\mathcal{CV},$ then
\begin{equation}\label{e2.1}
\frac{\Phi_2(\lambda, \alpha, k)}{2[1-\alpha + \Phi_2(\lambda,
\alpha, k)]} (f*g)(z) \prec g(z)
\end{equation}
where $0\leq k <\infty,$  $\lambda \geq -1,$ $0 \leq \alpha <1,$
with $\Psi_q(2,\lambda)$  given in \eqref{e2.3}
\begin{equation}\label{e2.4}
\Re \left ( f(z)\right ) > - \frac{1-\alpha +
\Phi_2(\lambda,\alpha, k)}{\Phi_2(\lambda,\alpha, k)},\,\,\,\, z
\in \mathbb{U}.
\end{equation}
The constant
 $ \frac{\Phi_2(\lambda,\alpha, k)}{2[1-\alpha + \Phi_2(\lambda,\alpha, k)}$  is the best estimate.
\end{theorem}
\begin{proof}
Let $f \in \mathcal{ST}^\lambda_q (k,\alpha)$ and suppose
that $g(z) = z + \sum\limits_{n=2}^{\infty} c_nz^n \in \mathcal{CV}.$ Then
\begin{align}\label{e2.5}
&\frac{\Phi_2(\lambda,\alpha, k)}{2[1-\alpha + \Phi_2(\lambda,
\alpha, k)]} (f*g)(z) \nonumber \\ &
\,\,\,\,\,\,\,\,\,\,\,\,\,\,\,\,\,\,\,\,\,\,\,\,\,\,\,=
\frac{\Phi_2(\lambda,\alpha, k)}{2[1-\alpha + \Phi_2(\lambda,
\alpha, k)]}\left ( z + \sum\limits_{n=2}^{\infty} c_na_nz^n\right
).
\end{align}
Thus, by Definition \ref{d2}, the subordination result holds true
if
\[
\left ( \frac{\Phi_2(\lambda,\alpha, k)}{2[1-\alpha +
\Phi_2(\lambda,\alpha, k)]} \right )_{n=1}^{\infty}
\]
is a subordinating factor sequence, with $a_1=1$. In view of Lemma
\ref{l1.1}, this is equivalent to the following inequality
\begin{equation}\label{e2.6}
\Re \left ( 1+
\sum\limits_{n=1}^{\infty}\frac{\Phi_2(\lambda, \alpha,
k)}{[1-\alpha + \Phi_2(\lambda,\alpha, k)]} a_n z^n \right ) >
0,\,\,\,\, z \in \mathbb{U}.
\end{equation}
Now, for $|z|=r<1,$ we have
\small{
\begin{eqnarray*}
&&\Re \left ( 1+ \frac{\Phi_2(\lambda, \alpha,
k)}{[1-\alpha +
\Phi_2(\lambda,\alpha, k)]}\sum\limits_{n=1}^{\infty} a_n z^n \right ) \\
&& = \Re\left ( 1+ \frac{\Phi_2(\lambda,\alpha,
k)}{[1-\alpha + \Phi_2(\lambda, \alpha,
k)]}z+\frac{\sum\limits_{n=2}^{\infty} \Phi_2(\lambda,\alpha, k)
a_n z^n}{[1-\alpha +
\Phi_2(\lambda,\alpha, k)]} \right ) \\
&& \geq 1 - \frac{\Phi_2(\lambda,\alpha, k)}{[1-\alpha +
\Phi_2(\lambda,\alpha, k)]}r  - \frac{\sum\limits_{n=2}^{\infty}
\Phi_n(\lambda, \alpha, k) a_n r^n}{[1-\alpha +
\Phi_2(\lambda,\alpha, k)]}\\
&& \geq 1 - \frac{\Phi_2(\lambda,\alpha, k)}{[1-\alpha +
\Phi_2(\lambda,\alpha, k)]}r - \frac{1-\alpha}{[1-\alpha +
\Phi_2(\lambda,\alpha, k)]}r \\  && > 0,\,\,\,\,
\end{eqnarray*}}
\
and  by noting the fact that $\Phi_n(\lambda,
\alpha, k)$ is increasing function for $n \geq 2.$ Thus we have
also made use of the assertion (\ref{e1.13}) of Theorem
\ref{tb1.1}. This evidently proves the inequality (\ref{e2.6}) and
hence also the subordination result (\ref{e2.1}) asserted by
Theorem \ref{t2.1}. The inequality (\ref{e2.4}) follows from
(\ref{e2.1}) by taking
$
g(z) = \frac{z}{1-z} = z + \sum\limits_{n=2}^{\infty}z^n \in \mathcal{CV}.
$
Next we consider the function
$
F(z) := z - \frac{1-\alpha}{\Phi_2(\lambda,\alpha, k)}z^2
$
where $0\leq k <\infty,$  $\lambda \geq -1,$ $0 \leq \alpha <1,$ and
$\Psi_q (2,\lambda)$ is given by (\ref{e2.3}). Clearly $F \in
\mathcal{ST}^\lambda_q (k,\alpha).$ For this function
(\ref{e2.1})becomes
\[
\frac{\Phi_2(\lambda,\alpha, k)}{2[1-\alpha + \Phi_2(\lambda,
\alpha, k)]} F(z) \prec \frac{z}{1-z}.
\]
It is easily verified that
\[
\min \left \{ \Re \left ( \frac{\Phi_2(\lambda, \alpha,
k)}{2[1-\alpha + \Phi_2(\lambda,\alpha, k)]} F(z) \right )\right
\} = - \frac{1}{2},\,\,\,\, z \in \mathbb{U}.
\]
This shows that the constant $\frac{\Phi_2(\lambda, \alpha,
k)}{2[1-\alpha + \Phi_2(\lambda,\alpha, k)]}$ is  best possible.
\end{proof}


\end{document}